\documentclass[12pt]{elsarticle}
\usepackage[a4paper, total={6.1in, 9.5in}]{geometry}
\usepackage{amsmath,amssymb,amsthm}
\usepackage{color}
\usepackage{graphicx}
\usepackage{verbatim}

\makeatletter
\def\ps@pprintTitle{%
	\let\@oddhead\@empty
	\let\@evenhead\@empty
	\def\@oddfoot{\centerline{\thepage}}%
	\let\@evenfoot\@oddfoot}
\makeatother

\newtheorem{theorem}{Theorem}[section]

\newtheorem{lemma}[theorem]{Lemma}

\newtheorem*{remark}{Remark}

\newcommand{\lap}{\mathcal{L}}
\newcommand{\etal}{\mbox{\emph{et al.\ }}}
\newcommand{\conv}{\ast}
\newcommand{\cD}{{}^c\!D}
\newcommand{\hD}{\!D^{\alpha,\gamma}}
\newcommand{\G}{{\mathcal G}}

\begin{document}
\begin{frontmatter}
\title{Inverse source in two-parameter anomalous diffusion, \\
numerical algorithms and simulations over graded time-meshes}
	
	\author[kfupm,kmfurati]{Khaled M. Furati}
	\author[kfupm,kassem]{Kassem Mustapha}
	\author[kfupm,isarumi]{Ibrahim O. Sarumi}
	\author[calu]{ Olaniyi S. Iyiola}
	
	\address[kfupm]{King Fahd University of Petroleum \& Minerals \\ Department of Mathematics \& Statistics
		\\ Dhahran 31261, Saudi Arabia}
	
	\address[kmfurati]{kmfurati@kfupm.edu.sa}
	\address[kassem]{kassem@kfupm.edu.sa}
	\address[isarumi]{isarumi@kfupm.edu.sa}
	\address[calu]{Department of Mathematics, Computer Science and Information Systems\\ California University of Pennsylvania \\
    California, PA, USA\\ iyiola@calu.edu}

	
\begin{abstract}
We consider an inverse source two-parameter sub-diffusion model subject to a 
nonlocal initial condition. 
The problem models several physical processes, among them are the microwave heating and 
light propagation in photoelectric cells.
A bi-orthogonal pair of bases is employed to construct a series representation of the solution and
a Volterra integral equation for the source term.
We develop a numerical algorithm for approximating the unknown time-dependent source term. 
Due to the singularity of the solution near $t=0$, a graded mesh is used to improve the convergence rate. 
Numerical experiments are provided to illustrate the expected analytical order of convergence.
\end{abstract}


\end{frontmatter}

\section{Introduction}
The main focus in this paper is to solve numerically the two-parameter time fractional diffusion inverse source problem: 
\begin{equation}\label{eq:problem}
\begin{aligned}
& D^{\alpha,\gamma} u(x,t)- u_{xx}(x,t) = w(t)h(x,t),
\qquad 0<x<1,\quad 0<t\leq T,\\
& I^{1-\gamma}u(x,t)|_{t=0} = g(x),\qquad \qquad 0 \leq x \leq 1,\\
& u(0,t)=u(1,t),\quad u_x(1,t)=0,\qquad 0 < t \leq T,
\end{aligned}
\end{equation}
with fractional exponents $0< \alpha \leq \gamma \leq 1$. The functions $h$ and $g$ are given,
while the solution $u$ and the inverse source term $w$ need to be determined.
Thus, and for well-posedness, we have to impose an over-determination condition, defined by
\begin{equation}\label{eq:over}
\int_0^1 u(x,t)dx = q(t), \qquad t\in[0,T].
\end{equation}
In the above model problem, the two-parameter fractional derivative operator $D^{\alpha,\gamma}$ is defined by
$$
D^{\alpha,\gamma} y(t)= D^\alpha \left(y(t) - \frac{I^{1-\gamma} y(0)}{\Gamma(\gamma)} \, t^{\gamma -1}\right),
$$
where $I^\mu$ and $D^\mu$ are the Riemann-Liouville fractional integral and derivative, respectively.
That is,  for $t>0$ and $0<\mu<1$, with $\omega_\mu(t):=t^{\mu-1}/\Gamma(\mu),$
$$
I^{\mu}  y(t)=\int_{0}^{t}\omega_\mu(t-\tau)
y(\tau) \, d\tau \quad {\rm and} \quad D^\mu y(t) =
\frac{d}{dt} I^{1-\mu} y(t).
$$

When $I^{1-\gamma} y$ is absolutely continuous, then $\hD y = I^{\gamma-\alpha} D^\gamma y = I^{\gamma-\alpha} D I^{1-\gamma} y$.
Hence, when  $\gamma = \beta(1-\alpha)+\alpha$, $0<\beta\leq 1$, $\hD$ reduces to the derivative introduced by Hilfer in \cite{hilfer:00evolution}.
Moreover, $\!D^{\alpha,\alpha} =D^\alpha $ and $\!D^{\alpha,1}=\cD^\alpha:= I^{1-\alpha} D$.  
Thus,  $\hD$ for $\alpha\le \gamma \le 1$ is considered as an interpolant between the Riemann-Liouville fractional derivative $D^\alpha$ and the Caputo fractional derivative $\cD^\alpha$.

Fractional PDEs open up new possibilities for robust mathematical modeling of physical processes 
that exhibit non-classical (non-Gaussian L\'evy and non-Markovian processes; and non-Brownian transport phenomena)
diffusion-dispersion. More precisely, fractional calculus provides a powerful tool for modeling a variety 
of nonlocal and memory-dependent  phenomena.  
Such phenomena are recognized in many areas such as nanotechnology \cite{baleanu:10b}, control theory of dynamical systems \cite{caponetto:10b,monje:10b}, viscoelasticity \cite{mainardi:10b,wang:18}, anomalous transport and diffusion \cite{klags:08b}, random walk dynamics \cite{michelitsch:17}, electrophysiology \cite{cusimano:18}, image processing \cite{kleefeld:18}, and flow in porous media \cite{amir:18}. Some other physical and engineering processes are given in \cite{ortigueira:11b,petras:11b} and more applications can be found in the surveys in \cite{hilfer:00b, kilbas:06b, podlubny:99b}.
In particular, fractional models are increasingly adopted for processes with anomalous diffusion \cite{adams:92,metzler:00,zhou:03}. The featured role of the fractional derivatives is mainly due to their non-locality nature which is an intrinsic property of many complex systems \cite{hilfer:00}.

The problem \eqref{eq:problem} models several physical processes, among them is the microwave heating. When a material is irradiated with microwaves, the absorption of electromagnetic energy within the material increases its temperature.  This effect can be modeled by replacing the electromagnetic power dissipation term in the nonlinear diffusion equation by an unknown equivalent internal heat source of the form $w(t) h(x,t)$ \cite{chen:97}.
The quantity $q(t)$ in \eqref{eq:over} represents the total absorbed energy due to the externally applied energy or it represents the total mass in the diffusion process of a chemical \cite{Cannon:1987, Gumel:1999, Hu:2007}.
Knowing the local conversion rate of microwave energy $h(x,t)$, determining the source term $w(t)$ gives an idea of how to control the external energy.

Inverse problems associated with models for anomalous diffusion processes arise in many applications.
These problems include determining the initial conditions, boundary conditions, diffusion, fluxes and potential coefficients, fractional orders, and source terms. Such problems are in general ill-posed and some additional requirement or measurement are provided to make them well-posed. More details are given in \cite{jin:15}.

When $\gamma<1$, the  non-local initial condition  $I^{1-\gamma}u(x,t)|_{t=0}=g(x)$ 
may lead to unbounded solution $u$ of the model  problem \eqref{eq:problem} near $t=0$, 
this will definitely  increase the level of complexity. 
However, when $\gamma=1$, (one parameter Caputo derivative), 
$I^{1-\gamma}u(x,t)|_{t=0}$ reduces to the standard (local) initial condition 
$u(x,t)|_{t=0}$. In this case, the inverse time-dependent source problems for fractional diffusion equation have been investigated under various initial, boundary and over determination conditions.

For unbounded spatial domains, {\"O}zkum \etal \cite{ozkum:13} used Adomian decomposition method to determine $w(t)$ assuming that $h(x,t)=1$ (spatial variable diffusivity was allowed). Later, with $h(x)$ in place of $h(x,t)$, Yang \etal \cite{yang:15} used the Fourier regularization method to obtain an a priori error estimate between the exact solution and its regularized approximation. 

For bounded domains and for $h=h(x),$ Sakamoto and Yamamoto \cite{sakamoto:11} used eigenfunction expansions to prove a stability result for the inverse source problem of determining $w(t)$, with $Lu$ in place of $-u_{xx}$ where $L$ is a linear symmetric elliptic operator. Wei and Zhang \cite{wei:13} solved numerically a Volterra integral equation (VIE) for   $w$ using  boundary element method combined with a generalized Tikhonov regularization. For a different given data, Aleroev \etal \cite{aleroev:13}, the Banach fixed-point theorem was used to prove the existence and uniqueness of $w.$
For different given data and with $h=h(x,t),$ a similar framework was considered by Ismailov and {\c{C}}i{\c{c}}ek \cite{ismailov:16}. Furthermore, Demir \etal \cite{demir:15} recovered $w$ by introducing input-output mappings and proved that their distinguishability holds under additional measurement data at a boundary point.

In all aforementioned cited works above, the problems considered involve the Caputo derivative together with the classical initial conditions.   However, in problem \eqref{eq:problem}, we consider a two parameter fractional derivative, of which, the Caputo and Riemann-Liouville derivatives are special cases, subject to nonlocal non-self adjoint boundary conditions and a nonlocal initial condition.
Unlike the space-dependent source problems considered in \cite{furati:14, furati:17}, 
this model gives rise to a VIE of the second kind for the source term $w$.
This equation cannot be solved analytically due to the presence of a complicated 
weakly singular kernel as well as a right-hand side that involves a two parameter 
fractional derivative plus an infinite series, see \eqref{eq:w-VIE}. 
So, a numerical scheme based on the discontinuous 
collocation method is developed to approximate $w$ by $\tilde w$.

Although we focused on finding a piecewise linear polynomial solution $\tilde w$, 
our approach can be extended to high-order polynomial solutions. 
For a smooth solution $w$, (that is, for smooth kernel $E$, smooth source $G$, 
and smooth coefficient $H$ in \eqref{eq:w-VIE}), the proposed scheme is second-order accurate globally. 
However, the solution $w$ has singularity near $t=0$ due to the weak singularity in $E$ and because the source term $G$ is generally not bounded near $t = 0$. 
Thus, to achieve an optimal global $O((\Delta t)^2)$ error 
($\Delta t$ is the maximum time step-size mesh element), we employ a non-uniform graded time meshes 
that based on concentrating the mesh elements near $t=0$ to compensate for the singular behavior of $w$, 
\cite{brunner:99, cao:03, mustapha:13, mustapharyan:13}. 
The existence and uniqueness of the collocation solution $\tilde w$ and the error bounds over graded time meshes are discussed. 
The numerical source $\tilde w$ is used to approximate the solution $u$ of \eqref{eq:problem}.

The rest of the paper is organized as follows.
In section \ref{sec:fourier}, we introduce the bi-orthogonal basis and state some preliminary results. 
Section \ref{sec:w} focuses on deriving the integral equation for the source term $w$ and on proposing a numerical algorithm for the numerical approximation of $\tilde w$. 
Existence, uniqueness and error analysis are investigated. 
Section \ref{sec: approximation of u} is devoted to seek an approximate solution of $u$ using the approximate source term $\tilde w$. 
We ended the paper with some simulations in section \ref{sec: simulation}.

\section{Series Representations}
\label{sec:fourier}
As in \cite{aleroev:13}, the non-self adjoint boundary conditions lead to the bi-orthogonal pair of bases $\{\phi_{1,0}, \phi_{1,i}, \phi_{2,i} \}_{i=1}^\infty$ and $\{\psi_{1,0}, \psi_{1,i}, \psi_{2,i} \}_{i=1}^\infty$
for the space $L^2(0,1)$: with $\lambda_i = 2\pi  i$, 

\begin{equation}
\label{eq:phi-bases}
\phi_{10}(x) = 2, \qquad \phi_{1i}(x) = 4(1-x)\sin \lambda_i x,
\qquad
\phi_{2i}(x) = 4\cos \lambda_i x,
\end{equation}

\begin{equation}
\label{eq:psi-bases}
\psi_{10}(x) = x, \qquad \psi_{1i}(x) = \sin \lambda_i x,
\qquad \psi_{2i}(x) = x\cos \lambda_i x.
\end{equation}
We consider a series solution of the form  
\begin{equation}
\label{eq:u-series}
u(x,t) =
u_{10}(t) \, \phi_{10}(x) + \sum_{\substack{i=1\\k=1,2}}^{\infty} u_{ki}(t) \, \phi_{ki}(x),
\end{equation}
where
\begin{equation}
\label{eq:hn}
u_{10}(t) = \langle u, \psi_{10} \rangle,\qquad u_{ki}(t) =
\langle u, \psi_{ki} \rangle, \quad k=1,2, \quad i=1,2,\dots.
\end{equation}
Here, $\langle \cdot, \cdot \rangle$ denotes the $L^2(0,1)$ inner product.
Similarly, we let $g_{ki}$ and $h_{ki}(t)$ denote the series coefficients of $g(x)$ and $h(x,t)$ with respect to the basis functions in \eqref{eq:phi-bases}. Following \cite{furati:17}, we assume that the functions $g$ and $h$ are sufficiently regular to guarantee the convergence of the series solution. 

Substituting the series expansions of $u$ and $h$ in \eqref{eq:problem}, we have
\begin{eqnarray}
\label{eq:u1}
D^{\alpha,\gamma} u_{1i}(t) + \lambda_i^2 u_{1i}(t) &=& w(t)h_{1i}(t),\qquad \qquad i \geq 0,
\\
\label{eq:u2}
D^{\alpha,\gamma} u_{2i}(t) + \lambda_i^2 u_{2i}(t) +
2\lambda_i u_{1i}(t)&=& w(t)h_{2i}(t), \qquad \qquad i \geq 1,
\end{eqnarray}
with $\lambda_0 :=0$. 
Moreover, from the initial condition in \eqref{eq:problem}, we obtain the initial conditions 
\begin{equation}
\label{eq:ic}
I^{1-\gamma} u_{10} (0) = g_{10}, \qquad
I^{1-\gamma} u_{ki} (0) = g_{ki},\qquad k=1,2, \quad i\ge 1.
\end{equation}
Using Laplace transform and its inverse, the solution of these initial value problems are
\begin{equation}
\label{eq:u1n}
u_{1i}(t) = \left(w\,h_{1i}\right)(t) \conv  \Theta^\alpha_i(t) +
g_{1i} \Theta^\gamma_i(t), \qquad i \geq 0,
\end{equation}
and
\begin{equation}
	\label{eq:u2n}
	\begin{aligned}
	& u_{2i}(t)  =
	\left(w\,h_{2i}\right)(t) \conv  \Theta^\alpha_i(t)
	+ g_{2i} \Theta^\gamma_i(t)
	\\ &\quad -
	2 \lambda_i \Big[
	\left(w\,h_{1i}\right)(t)\conv (t^{2\alpha - 1}
	E_{\alpha,2\alpha}^2(-\lambda_i^2 t^\alpha)) +
	g_{1i}t^{\sigma} 
	E_{\alpha,\alpha+\gamma}^2(-\lambda_i^2 t^\alpha) 
	\Big], \quad i \geq 1,
	\end{aligned}
\end{equation} 
where the convolution 
$$
f(t)\conv g(t) := \int_0^t f(\tau) \, g(t-\tau)\, d\tau~~{\rm and}~~
\Theta^\gamma_i(t) := t^{\gamma-1} E_{\alpha,\gamma}(-\lambda_i^2 t^\alpha).$$
Here, the generalized Mittag-Leffler function \cite{prabhakar:71} is defined by 
\begin{equation}
\label{eq:GMLF1}
E_{\alpha,\beta}^\rho (t) = \sum_{k=0}^\infty
\frac{\Gamma(\rho+k)}{\Gamma(\rho) \, \Gamma(\alpha k + \beta)} \, \frac{t^k}{k!}.
\end{equation}

Therefore, once the source term $w(t)$ is determined, the series coefficients of the solution $u(x,t)$ can be computed.
In the next section,  we discuss the determination and approximation of $w(t)$.
\section{Determination and approximation of $w$.}
\label{sec:w}

This section is devoted to discuss the process of approximating the source term $w(t)$ in \eqref{eq:problem}. 
It turns out that $w$ satisfies a VIE of the second kind with variable coefficients. 
Due to the complexity in solving this integral equation analytically, the numerical solution of $w$ via a discontinuous collocation method 
is investigated.

\subsection{Integral equation}
Integrating the first equation  in \eqref{eq:problem} over the spatial domain $[0,1]$  
and using the given boundary conditions, we obtain for $0<t<T$ that
\begin{equation}
\label{eq:HW}
 H(t) w(t) =  \int_0^1 D^{\alpha,\gamma} u(x,t)dx + u_x(0,t), \quad {\rm where}~~H(t) = \int_0^1 h(x,t)dx.
\end{equation}
From condition \eqref{eq:over},
\begin{eqnarray}
\int_0^1 D^{\alpha,\gamma} u(x,t)dx = D^{\alpha,\gamma} \int_0^1 u(x,t)dx = 
D^{\alpha,\gamma} q(t),
\end{eqnarray}
and thus,
\begin{equation}
\label{eq:problem-integral}
 H(t) w(t) =  D^{\alpha,\gamma} q(t) + u_x(0,t), \qquad 0<t<T.
\end{equation}
For the second term on the right hand side of \eqref{eq:HW}, the series 
representation of $u(x,t)$ in \eqref{eq:u-series} and the associated  coefficients \eqref{eq:u1n} yield
\begin{align*}
u_x(0,t) &= 4 \sum_{n=1}^\infty \lambda_i u_{1i}(t)=
4  \sum_{i=1}^\infty \lambda_i \Big\{
g_{1i} \, \Theta^\gamma_i(t) + \left( w\,h_{1i}\right)(t)
\conv  \Theta^\alpha_i(t) \Big\}
\\ &=
4\sum_{i=1}^\infty \lambda_i
g_{1i} \, \Theta^\gamma_i(t) + 
\int_0^t w(\tau) \sum_{i=1}^\infty 4 \lambda_i \, h_{1i}(\tau)\,  
\Theta^\alpha_i(t-\tau) \, d \tau.
\end{align*}
Inserting this in \eqref{eq:problem-integral} amounts to the following VIE of the second kind:
\begin{equation}
\label{eq:w-VIE}
H(t) \, w(t)-\int_0^t E(t,\tau) \, w(\tau) \, d\tau= G(t),
\end{equation}
where
\begin{equation}
\label{eq:E}
E(t,\tau) = 4 \sum_{i=1}^\infty \lambda_i \, h_{1i}(\tau) \, \Theta^\alpha_i(t-\tau),
\end{equation}
and
\begin{equation}
\label{eq:G}
G(t) = D^{\alpha,\gamma} q(t) + 4 \sum_{i=1}^\infty \lambda_i \, g_{1i} \, \Theta^\gamma_i(t).
\end{equation}

Note that the kernel $E$  can be written in the form
\begin{equation}
\label{eq: bound of E}
E(t,\tau) = (t-\tau)^{\alpha-1} \, \tilde{E}(t,\tau), \qquad
\tilde{E}(t,\tau) = 4 \sum_{i=1}^\infty \lambda_i \, h_{1i}(\tau) \,
E_{\alpha,\alpha} \left(-\lambda_i^2 (t-\tau)^\alpha \right).
\end{equation}
Since $|E_{\alpha,\alpha} \left(-\lambda_i^2 s^\alpha \right)|\le \dfrac{C}{1+\lambda_i^2 s^\alpha}$, 
\begin{equation}
\label{eq: bound of E1}
|\tilde{E}(t,\tau)| \le C \sum_{i=1}^\infty \lambda_i^{-1} \, h_{1i}(\tau) \le C,
\end{equation}
where the second inequality is valid provided that $h(x,t)$ has a finite number of 
Fourier-modes or is sufficiently regular 
(for instance, $h_{1i}(\tau)\le C\lambda_i^{-2-\epsilon}$ for some positive $\epsilon$), 
see \cite{furati:17} for more details. As a consequence, $\tilde{E}(\cdot,\cdot)$ 
is continuous on $[0,T]\times [0,T]$. Also,  for sufficiently regular functions $g$ and $q$, 
one can check that the function $G$ 
is continuous on  $[\epsilon,T]$ for $0<\epsilon<T.$ This guarantees the existence
and uniqueness of the continuous solution $w_\epsilon$ of \eqref{eq:w-VIE} on the interval 
$[\epsilon,T]$, for more details, we refer the reader to \cite[Chapter 6]{brunner:04b} or
 \cite[Section 3.4]{linz:85}. In the limiting case, $\epsilon$ approaches $0$, the existence 
and uniqueness of the continuous solution $w$ of \eqref{eq:w-VIE} on $(0,T]$ follows.
 
The numerical solution of  \eqref{eq:w-VIE} via a discontinuous collocation method will be investigated in the next subsection. The upper bound in \eqref{eq: bound of E1} is used below to derive the error estimates for our discretization. 

\subsection{Approximation of $w(t)$}
\label{sec:w-approx}
We intend to approximate the solution $w$ of \eqref{eq:w-VIE} by piecewise 
linear polynomials via a robust discontinuous collocation method. 
To do so, we introduce a  time partition of the interval $[0,T]$ given by the graded nodes 
$t_n=(n/N)^\delta T$ (with $\delta \ge 1$) for $0\le n\le N$. 
Let $I_n=(t_{n-1},t_n]$ and $\Delta t_n=t_n-t_{n-1}$, $1\le n\le N$, 
with $\Delta t:=\max_{1\le n\le N}\Delta t_n$. 

Let $S$ be the finite dimensional space of linear polynomials on each time mesh element $I_n$, $1 \leq n \leq N$. 
The set of collocation (grid) points is
$$
X= \Big\{ t_{n,j}:= t_{n-1}+\xi_j \Delta t_n, \quad 1\le j\le 2,\quad n=1,\dots, N
\Big\}, \qquad 0 < \xi_1< \xi_2< 1.
$$ 
One option is to choose $\xi_1$ and $\xi_2$ to be  the Gauss quadrature points.

The discontinuous collocation solution ${\tilde w} \in S$ is now defined by
\begin{equation}
\label{eq:tilde w-equation}
H(t) \, \tilde w(t)-\int_0^t E(t,\tau) \, \tilde w (\tau) \, d\tau= G(t),
\qquad t\in X.
\end{equation}
Alternatively, this scheme can be rewritten as: for $j=1,2,$ and  $n = 1, \dots, N,$ 
\begin{equation}
\label{eq:VIE-discrete}
H(t_{n,j}) \, \tilde w(t_{n,j}) - \int_0^{t_{n,j}} E(t_{n,j},\tau) \, \tilde w(\tau) \, d\tau 
=
G(t_{n,j})\,.
\end{equation}
On each subinterval $I_n$, we will base the computational form of $\tilde w$ on the 
Lagrange basis functions with respect to the collocation parameters $\{\xi_1,\xi_2\}$. 
That is,
\begin{equation}
\label{eq: Lagrange representation}
\tilde w(t) = \tilde w(t_{n,1}) L_{n,1}(t) + \tilde w(t_{n,2}) L_{n,2}(t),
\quad t\in I_n,
\end{equation}
where for each fixed $n$, $\{L_{n,j}\}_{j=1,2}$ are the local Lagrange basis functions 
associated with the collocation points $\{t_{n,j}\}_{j=1,2}$ corresponding to the interval $I_n$ defined by
\begin{equation}
\label{eq:basis}
L_{n,1}(t)=\frac{t_{n,2}-t}{\zeta_n}\quad{\rm and}\quad L_{n,2}(t)=
\frac{t-t_{n,1}}{\zeta_n},\quad \zeta_n=t_{n,2}-t_{n,1}.
\end{equation}

To handle the integral term in \eqref{eq:VIE-discrete}, we  use the definition $\tilde w$ 
given in \eqref{eq: Lagrange representation} and obtain,
\begin{align*}
 \int_0^{t_{n,j}} E(t_{n,j},\tau) \, \tilde w(\tau) \, d\tau
 & =
\sum_{s=1}^2\sum_{m=1}^n \tilde w(t_{m,s}) 
\int_{t_{m-1}}^{\min\{t_{n,j},t_m\}} E(t_{n,j},\tau) \, L_{m,s}(\tau) \, d\tau
\\ & =
\sum_{s=1}^2  \sum_{m=1}^n  a_{m,s}(t_{n,j}) \, \tilde w(t_{m,s}),
\end{align*}
where, for $1\le m\le n,$ 
\[a_{m,s}(t_{n,j}) = \int_{t_{m-1}}^{\min\{t_{n,j},t_m\}} E(t_{n,j},\tau) \, L_{m,s}(\tau) \, d\tau.\]

Thus the numerical scheme in \eqref{eq:VIE-discrete} can be written as
\begin{equation}
\label{eq:wnj}
H(t_{n,j}) \,\tilde w(t_{n,j}) - \sum_{s=1}^2 a_{n,s}(t_{n,j}) \, \tilde{w}^{n,s} = F^{n,j},
\qquad j=1,2, \quad n = 1, \dots, N,
\end{equation}
with
$$
F^{n,j} = G(t_{n,j})+ \sum_{s=1}^2  \sum_{m=1}^{n-1}  a_{m,s}(t_{n,j}) \, \tilde w(t_{m,s}).
$$
Therefore, for each $n$, we have to solve the $2\times 2$ linear system
\begin{equation}
\label{eq: matrix form}
({\bf H}^n - {\bf B}^n) \left[ \begin{matrix} \tilde{w}(t_{n,1}) \\ \tilde{w}(t_{n,2}) 
\end{matrix} \right]= {\bf F}^n:=\left[ \begin{matrix} F^{n,1} \\ F^{n,2} \end{matrix} \right],
\end{equation}
with
$${\bf H}^n = 
\left[ \begin{matrix} H(t_{n,1}) & 0 \\  0 & H(t_{n,2}) \end{matrix} \right], \qquad
{\bf B}^n = 
\left[ \begin{matrix} a_{n,1}(t_{n,1}) & a_{n,2}(t_{n,1}) \\  a_{n,1}(t_{n,2}) & a_{n,2}(t_{n,2}) 
\end{matrix} \right].
$$
Since our discontinuous collocation scheme \eqref{eq:tilde w-equation}  amounts to a $2$-by-$2$ 
linear system on each time mesh element $I_n$, the existence of the solution $\tilde w$ follows from its uniqueness. 
To guarantee the uniqueness of $\tilde w$, it is clear from  \eqref{eq: matrix form} that 
we need the matrix ${\bf H}^n - {\bf B}^n$ to be non-singular. 
However, the kernel is weakly singular of order $\alpha-1$, then for any $0<\alpha \leq 1$, 
there exists $k_\alpha > 0$ such that the non-singularity of the matrix  ${\bf H}^n - {\bf B}^n$ 
can be assured  whenever $\Delta t<k_\alpha$, provided the matrix $H$ is non-singular.   

The collocation scheme \eqref{eq:VIE-discrete} is yet fully discrete due to the integrals 
in the entries of the matrix ${\bf B}^n$ and the vector ${\bf F}^n$. 
To compute these entries, we approximate the Fourier coefficient $h_{1i}(t)$, $i=1,2,\dots$ 
by its midpoint average value on each time subinterval $I_n$ and the remaining part of the 
integral can be computed exactly. 
This approximation preserves the second order of accuracy (over nonuniform meshes) of the 
collocation scheme in the presence of the weakly singular kernel $E$. 
Explicitly, for $1\le n,m\le N$, and for $j,s \in \{1,2\}$,
\begin{align*}
a_{m,s}(t_{n,j}) &=
4 \int_{t_{m-1}}^{\min\{t_{n,j},t_m\}} \sum_{i=1}^\infty \lambda_i \, h_{1i}(\tau) \,
\Theta^\alpha_i(t_{n,j}-\tau) \, L_{m,s}(\tau) \, d\tau
\\ &\approx
4  \sum_{i=1}^\infty \lambda_i \, h_{1i}^{m-1/2} \, \int_{t_{m-1}}^{\min\{t_{n,j},t_m\}} 
\Theta^\alpha_i(t_{n,j}-\tau) \, L_{m,s}(\tau) \, d\tau
\\ &=
4  \sum_{i=1}^\infty \lambda_i \, h_{1i}^{m-1/2} \, 
\int_{t_{m-1}}^{\min\{t_{n,j},t_m\}} 
(t_{n,j}-\tau)^{\alpha-1} \, E_{\alpha,\alpha}(-\lambda_i^2 (t_{n,j}-\tau) \, 
L_{m,s}(\tau) \, d\tau
\\  &=
4  \sum_{i=1}^\infty \frac{1}{\lambda_i} \, h_{1i}^{m-1/2} \, 
\Psi^\alpha(t_{n,j},\tau, \lambda_i^2, L_{m,s})\Big|_{t_{m-1}}^{\min\{t_{n,j},t_m\}},
\end{align*}
where the integral formula  in the lemma below  is used in the last equality. 

\begin{lemma}\label{lem: formula} 
For any linear polynomial   $L(t) = \nu t + c$, we have 
\begin{equation}\label{eq:psi} \lambda \int_a^b (t-\tau)^{\alpha-1} \, E_{\alpha,\alpha} (-\lambda(t-\tau)^\alpha ) \, L(\tau) \, d\tau  =
\Psi^\alpha(t,\tau, \lambda, L)\Big|_{\tau=a}^{\tau=b},
\end{equation}
where
\[
\Psi^\alpha(t,\tau, \lambda, L) =
E_\alpha(-\lambda(t-\tau)^\alpha) \, L(\tau)
+ \nu \, (t-\tau) \, E_{\alpha,2}(-\lambda(t-\tau)^\alpha).\]
\end{lemma}
\begin{proof}
Recall  the following formulas:  if $\alpha, \beta, \lambda \in \mathbb{C} \text{ such that } Re(\alpha) > 0 \text{ then }$
\begin{align*}\lambda t^\alpha \; E_{\alpha,\alpha+\beta}(-\lambda t^\alpha) &=
\frac{1}{\Gamma(\beta)} - E_{\alpha,\beta}(-\lambda t^\alpha),\\
\int (t-\tau)^{\beta-1} E_{\alpha,\beta}^\rho ( - \lambda (t-\tau)^\alpha ) \, d\tau
& =
-(t-\tau)^\beta E_{\alpha,\beta+1}^\rho (-\lambda (t-\tau)^\alpha) + C.
\end{align*}
Integrating by parts, and then using the second formula followed by the first one, we notice that  
\begin{align*}
\lambda \int & (t-\tau)^{\alpha-1} \, E_{\alpha,\alpha} (-\lambda(t-\tau)^\alpha ) \, L(\tau) \, d\tau 
\\& =
- \lambda (t-\tau)^\alpha E_{\alpha,\alpha+1}(-\lambda(t-\tau)^\alpha) \, L(\tau)
+ \lambda \nu \int (t-\tau)^\alpha E_{\alpha,\alpha+1}(-\lambda(t-\tau)^\alpha) \, d\tau
\\& =
- \lambda (t-\tau)^\alpha E_{\alpha,\alpha+1}(-\lambda(t-\tau)^\alpha) \, L(\tau)
- \lambda \nu \, (t-\tau)^{\alpha+1} E_{\alpha,\alpha+2}(-\lambda(t-\tau)^\alpha)
\\& =
\left[ E_\alpha(-\lambda(t-\tau)^\alpha) -1 \right] L(\tau)
+ \nu \, (t-\tau) \left[ E_{\alpha,2}(-\lambda(t-\tau)^\alpha) - 1 \right].
\end{align*}
Since $L(\tau) + \nu (t-\tau) \Big|_{\tau=a}^{\tau=b} = 0$, the desired result follows.
\end{proof}

\subsection{Error analysis}
\label{sec: error analysis}
As mentioned earlier, for VIEs with a smooth kernel and source term, the collocation scheme is $O((\Delta t)^2)$ accurate.
However, due to the lack of regularity of the continuous solution $w$ (has singularity near $t=0$) caused by the presence of the weakly singular kernel $E$ and the nonsmooth source term $G$, such order of accuracy is not feasible over uniform meshes. For this reason and to improve the convergence rates, a graded mesh with time nodes  $t_n= (n/N)^\delta T$ ($\delta\ge 1$) is employed \cite{mustapha:13, mustapharyan:13} to compensate for the singular behavior of $w$ near $t=0$.  

Next, a concise proof of the error estimates from the discontinuous collocation discretization over the graded mesh $t_n=(n/N)^\delta T$ (with $\delta \ge 1$) is given. We impose the following typical assumption on the derivative of the continuous solution $w$; for some $0<\sigma<1$ and for $t>0,$  
\begin{equation}
\label{eq: reg}
|w(t)|\le a(1+t^\alpha)+bt^{\gamma-1} \quad {\rm and} \quad |w'(t)|+t|w''(t)| \le C\,t^{\sigma-1},
\end{equation}
for some constants $a, b$.

\begin{remark}
Under reasonable assumptions on the given data, the above regularity conditions can be verified, 
where typically $\sigma=\alpha+\gamma-1$ and so, we require $\alpha+\gamma>1$. 
Another note, in the presence of the two-parameter fractional derivative $D^{\alpha,\gamma}$ ($\gamma<1$), 
the solution $w$ has a strong singularity near $t=0$.
More precisely, $w$  is not bounded near $t=0$ and $w'$ is not integrable on $[0,T]$ (unless $\sigma>0$).  
Therefore, efficient approximations of $w$ on the first sub-interval $I_1$  require more effort. 
One way, is to inherit the order of singularity (that is due to $\gamma$) in the approximate solution, 
for instance, approximate $w$ by the function $a_1+a_2 t^{\gamma-1}$ on $I_1$.
Investigating this is beyond the scope of this paper and it will be a topic of future research. 
On the other hand, it is worth to mention that for the case of Caputo derivative, $\gamma=1$,
$w$ is uniformly bounded on $[0,T]$ and $w'$ is integrable on $[0,T]$. 
\end{remark}
  
Decomposing the error  $w -\tilde w$ as 
$$
w-\tilde w= (w-P w)+(P w-\tilde w) =: \eta_1+\eta_2,
$$
where the comparison function  $P w \in S$ interpolates $w$ at the composite Gauss-quadrature nodes, that is, 
\begin{equation}
\label{eq: interpolation}
P w(t) = w(t_{n,1}) L_{n,1}(t) + w(t_{n,2}) L_{n,2}(t),
\quad t\in I_n, \quad 1\le n\le N.
\end{equation}
It is known that 
$$
\|\eta_1\|_{L^\infty(I_n)}\le C (\Delta t_n)^{j-1}\int_{I_n}|w^{(j)}(t)|\,dt,\qquad j=1, 2.
$$
Hence, by the regularity assumption \eqref{eq: reg} on $w$ we have
$$
\|\eta_1\|_{L^\infty(I_1)}\le C \int_{I_1}t^{\sigma-1}\, dt < Ck_1^{\sigma}
\le C(\Delta t)^{\sigma\delta}.
$$
Furthermore, for $n\ge 2,$ using the mesh properties; $t_n\le C t_{n-1}$ and 
$\Delta t_n\le k\,t_n^{1-1/\delta},$ we notice that 
$$
\|\eta_1\|_{L^\infty(I_n)}\le C (\Delta t_n)^2t_n^{\alpha+\gamma-3}= 
C(\Delta t_n/t_n)^{\sigma\delta} (\Delta t_n/t_n)^{2-\sigma\delta}t_n^{\sigma} 
\le C(\Delta t)^{\sigma\delta},
$$
for $0< \sigma\delta\le 2.$ Therefore, for $1\le n\le N$, 
$$
\|\eta_1\|_{L^\infty(I_n)}\le C(\Delta t)^{\sigma\delta},~~{\rm for}~~1\le \delta\le 2/\sigma.
$$

To estimate the second term $\eta_2$, we notice from \eqref{eq:w-VIE}, \eqref{eq:tilde w-equation} 
and the definition of $P w$ that 
\begin{equation}
H(t)\eta_2(t)-\int_0^t E(t,\tau) \, \eta_2(\tau) \, d\tau= \int_0^t E(t,\tau) \, \eta_1(\tau) \, d\tau,
\qquad t\in X.
\end{equation}
Using $|E(t,\tau)|\le C(t-\tau)^{\alpha-1}$ for $t<\tau$ 
(see \eqref{eq: bound of E}--\eqref{eq: bound of E1}), after some manipulations 
and for a sufficiently small step size $k$,  the above equation amounts  to 
\[
|\eta_2(t_{n,j})|\le C\max_{1\le j\le n}\|\eta_1\|_{L^\infty(I_j)} \int_0^{t_{n,j}} E(t_{n,j},\tau)\, d\tau+ C\sum_{i=1}^{n-1}
\int_{t_{i-1}}^{t_i}(t_{n,j}-\tau)^{\alpha-1}\, d\tau\sum_{\ell=1}^2 |\eta_2(t_{i,\ell})|,
\]
for $1\le n\le N$ and for $j=1,\,2.$ Thus, an application of the weakly singular 
discrete Gronwall inequality \cite[Theorem 6.1]{dixon:86} gives 
\[
|\eta_2(t_{n,j})|\le C\max_{1\le j\le n}\|\eta_1\|_{L^\infty(I_j)} 
\int_0^{t_{n,j}} E(t_{n,j},\tau)\, d\tau\le C\max_{1\le j\le n}\|\eta_1\|_{L^\infty(I_j)}.
\]
Next, the expansion of $\tilde w$ given in \eqref{eq: Lagrange representation} and the 
achieved estimate of $\eta_1$ yield 
\[
\|\eta_2\|_{L^\infty(I_n)}\le C(\Delta t)^{\sigma\delta},~~{\rm for}~~1
\le \delta\le 2/\sigma,\quad{\rm for} ~1\le n\le N,
\]
and consequently, we attain the following error bound, for $\alpha+\gamma>1,$  
\[
\|w-\tilde w\|_{L^\infty(I_n)}\le C(\Delta t)^{\sigma\delta},~~{\rm for}~~1\le 
\delta\le 2/\sigma,\quad{\rm for} ~1\le n\le N.
\]
The above error bound shows that a global  optimal $O((\Delta t)^2)$ convergence rates can be recovered by choosing the graded mesh exponent $\delta$ to be $2/\sigma$.

\section{Approximation of $u$}
\label{sec: approximation of u}
This section is devoted to discuss the numerical approximation of the solution 
$u$ of the fractional model problem \eqref{eq:problem}, denoted by $\tilde u$. 
More precisely, we approximate the Fourier coefficients $u_{ki}$ for $k=1, 2$ 
in equations \eqref{eq:u1n} and \eqref{eq:u2n} by $\tilde u_{ki}$ 
using the (source term) collocation solution $\tilde w$. Thus,  
\begin{align*}
\tilde u_{1i}(t)
&\approx \left(\tilde w\,h_{1i}\right)(t)\conv  \Theta^\alpha_i(t)
+ g_{1i} \Theta^\gamma_i(t), \qquad i\geq 0,
\\
\tilde u_{2i}(t) &\approx 
\left(\tilde w\,h_{2i}\right)(t) \conv  \Theta^\alpha_i(t)
+ g_{2i} \Theta^\gamma_i(t)
\\ & \quad -
2 \lambda_i \left[
\left(\tilde w\,h_{1i}\right)(t)\conv (t^{2\alpha-1} E_{\alpha,2\alpha}^2(-\lambda_i^2 t^\alpha))
+ g_{1i}t^{\sigma} E_{\alpha,\alpha+\gamma}^2(-\lambda_i^2 t^\alpha) \right],
 \quad i \geq 1.
\end{align*}

Following the convention in the previous subsection,  we approximate the convolutions as follows:
\begin{align*}
\left(\tilde w\,h_{1i}\right)(t_n)\conv  \Theta^\alpha_i(t_n)
	 &\approx
	\sum_{m=1}^{n}
	h_{1i}^{m-1/2} \int_{t_{m-1}}^{t_m}
	\Theta_i^\alpha(t_n-\tau) \,
	\sum_{s=1}^2 \tilde w(t_{m,s}) \, L_{m,s}(\tau) \, d\tau
	\\ &=
	\sum_{m=1}^{n}
	h_{1i}^{m-1/2}
	\sum_{s=1}^2 \tilde w(t_{m,s}) \, \int_{t_{m-1}}^{t_m} \Theta_i^\alpha(t_n-\tau) \, L_{m,s}(\tau) \, d\tau.
\end{align*}
Hence, for $ i > 0$
\[\tilde u_{1i}(t_n) \approx \sum\limits_{m=1}^n \frac{h^{m-1/2}_{1i}}{\lambda_i^2} \sum\limits_{s = 1}^2 \tilde w(t_{m,s})\Psi^\alpha(t_n,\tau, \lambda_i^2, L_{m,s})\Big|_{t_{m-1}}^{t_m} + g_{1i} \Theta^\gamma_i(t_n).\]
In a similar fashion, the first two terms in  $\tilde u_{2i}$ can be computed (just replace $h_{1i}$ and $g_{1i}$ with $h_{2i}$ and $g_{2i}$, respectively,  in the above equation). 

It remains to compute $\tilde u_{10}$ and the second convolution term in  $\tilde u_{2i}.$   For $i = 0$,
\begin{align*}
	(\tilde wh_{10}\conv \omega_\alpha)(t_n) &\approx
	\sum_{m=1}^{n}
	h_{10}^{m-1/2} \int_{t_{m-1}}^{t_m}
	\tilde w(\tau)\,\omega_\alpha(t_n - \tau)\, d\tau
	\\ &\approx
	\sum_{m=1}^{n}
	h_{10}^{m-1/2}
	\sum_{s=1}^2 \tilde w(t_{m,s}) \, \int_{t_{m-1}}^{t_m} L_{m,s}(\tau) \omega_\alpha(t_n - \tau)\, d\tau,
\end{align*}
and thus, 
\begin{multline*}
	\tilde u_{10}(t_n) 
	=
	(wh_{10}\conv \omega_\alpha)(t_n) + g_{10}\omega_\gamma(t_n)
	\approx 	
	\sum_{m=1}^{n}
	h_{10}^{m-1/2} \sum_{s=1}^2 \tilde w(t_{m,s}) \times \\
	\Big(L_{m,s}(\tau)\, \omega_{\alpha+1}(t_n - \tau) + 
	\frac{(-1)^s}{\zeta_m}\omega_{\alpha+2}(t_n - \tau)\Big)\Big|^{t_{m-1}}_{t_m} + 
	g_{10}\omega_\gamma(t_n).
\end{multline*}
Similarly, we compute the second convolution term in  $\tilde u_{2i}$, and get 
\[
\left[(\tilde w\,h_{1i})(t_n) \conv \big(t_n^{2\alpha - 1} 
E_{\alpha,2\alpha}^2(-\lambda_i^2 t^\alpha_n)\big)\right]
\approx
\sum_{m=1}^{n}
h_{1i}^{m-1/2} \sum_{s=1}^2 \tilde w(t_{m,s}) \, \mathcal{T}_{\alpha,n}^{m}
\]
with
\begin{align*}
\mathcal{T}^{\alpha,n}_{m,s}
&:=
\int_{t_{m-1}}^{t_m} L_{m,s}(\tau) (t_n - \tau)^{2\alpha-1}
E_{\alpha,2\alpha}^2(-\lambda_i^2 (t_n - \tau)^\alpha) d\tau
\\ &=
-L_{m,s}(\tau)(t_n-\tau)^{2\alpha} E_{\alpha,2\alpha+1}^2
(-\lambda_i^2 (t_n-\tau)^\alpha)\Big|_{t_{m-1}}^{t_m}
\\ &\quad-
\frac{(-1)^s}{\zeta_m}(t_n-\tau)^{2\alpha+1} E_{\alpha,2\alpha+2}^2
(-\lambda_i^2 (t_n-\tau)^\alpha)\Big|_{t_{m-1}}^{t_m}.
\end{align*}
In the second equality, we integrated by parts and used the second formula in the proof of 
Lemma  \ref{lem: formula}.  
By Combining the above contributions, we complete the computation of $\tilde u_{2i}$.

\section{Numerical results}\label{sec: simulation}

This section is devoted to illustrate numerically the convergence of the proposed discontinuous 
collocation scheme for the VIE of the form \eqref{eq:w-VIE} and to compare between the analytical 
solution $u$ of \eqref{eq:problem} and its approximation $\tilde u$.
For this purpose and to be able to  compute $w$ and $u$ exactly, we consider problem \eqref{eq:problem} 
with $T=1$ and
\begin{align*}
g(x) &= 4\cos(4\pi x),
\\ 
q(t) &= \omega_{1+\alpha}(t), \\
h(x,t) &= \frac{\lambda_1^2-1}{\lambda_1^2} +  \frac{(1-x)\sin(2\pi x)}{ \lambda_1} + 4t\cos(2\pi x)+8\cos(4\pi x). 
\end{align*}
The only nonzero series coefficients of $h(x,t)$ and $g(x)$ are
\begin{equation}
\label{eq:example-coeff}
h_{10} = \frac{\lambda_1^2-1}{2\lambda_1^2}, \qquad	h_{11} = \frac{1}{4\lambda_1}, \qquad 
h_{21}(t) = t, \qquad h_{22}(t)=2, \qquad g_{22} = 1.
\end{equation}
Furthermore, for this choice of $h(x,t)$ and $q(t)$, we have
\begin{equation}
\label{ex:H}
H(t) = \int_0^1 h(x,t) \, dx =  2 \, \frac{\lambda_1^2-1}{2\lambda_1^2} + 
\frac{1}{\lambda_1^2} = 1 \quad {\rm and} \quad D^{\alpha,\gamma} q(t)= D^\alpha q(t) = 1.
\end{equation}
Inserting these terms  in  \eqref{eq:E} and \eqref{eq:G}, we reach 
$$
E(t,\tau)  = \Theta^\alpha_1(t-\tau) \quad {\rm and} \quad
G(t)  = D^{\alpha,\gamma} q(t) + \sum_{i=1}^\infty 4\lambda_i \, g_{1i} \, \Theta^\gamma_i(t) = 1.
$$
Therefore, the VIE \eqref{eq:w-VIE} reduces to  $w(t) - (w \conv \Theta^\alpha_1)(t) = 1$.
To find $w$, we take Laplace transform of both sides and use the formula 
\begin{equation}
\label{eq:lap-mlf-general}
\lap\{t^{\beta-1} \, E_{\alpha,\beta} \left(a t^\alpha \right)\}(s)
=\frac{s^{\alpha -\beta}}{s^\alpha - a}, \qquad \alpha, \beta>0.
\end{equation}
Let $W$ denotes Laplace transform of $w$, then
\begin{equation}
\label{eq:lap-w}
W(s) = \frac{1}{s} \frac{s^\alpha + \lambda_1^2}{s^\alpha + \lambda_1^2 - 1}
= \frac{1}{\lambda_1^2-1} \left[\frac{\lambda_1^2}{s} - 
\frac{s^{\alpha-1}}{s^\alpha+\lambda_1^2 -1} \right].
\end{equation}
By taking the inverse Laplace transform, the exact source term is given by
\begin{equation}
\label{w-example}
w(t) = \frac{1}{\lambda_1^2-1} \left[ 
\lambda_1^2 - E_\alpha ((1 - \lambda_1^2) t^\alpha) \right].
\end{equation}

Next, we find the series coefficients for the solution $u(x,t)$. By \eqref{eq:u1n}, \eqref{eq:u2n} and \eqref{eq:example-coeff},
\begin{equation}
\label{eq:u1i-example}
u_{1i}(t) = h_{1i} \, (w \conv \Theta^\alpha_i)(t)~~{\rm for}~~i = 0, 1,~~{\rm and}~~
u_{1i}=0~~{\rm for}~~ i \ge 2,\end{equation}
 and
\begin{equation}
\label{eq:u21}
u_{2i}(t)  = \begin{cases}(t w(t) ) \conv  \Theta^\alpha_1(t) -
\frac{1}{2}
w(t) \conv \big(t^{2\alpha - 1} E_{\alpha,2\alpha}^2(-\lambda_1^2 t^\alpha) 
\big),&\qquad i = 1,
\\ 
 2 (w \conv  \Theta^\alpha_2)(t)  + \Theta^\gamma_2(t),&\qquad i = 2,\\
0, &\qquad i \ge 3.\end{cases}\end{equation}
Therefore, the exact solution is given by
\begin{equation}
\label{eq: example exact solution}
u(x,t) = 2u_{10}(t) + 4(1-x) \sin (\lambda_1x)\, u_{11}(t) + 
4 \cos (\lambda_1x)\, u_{21}(t) + 4 \cos (\lambda_2x)\, u_{22}(t).
\end{equation}
We compute the convolutions in \eqref{eq:u1i-example} and \eqref{eq:u21} through the 
Laplace transform.
For $i\geq 0$, using the formula in \eqref{eq:lap-mlf-general} and  the achieved 
contribution in \eqref{eq:lap-w} yield
\begin{align*}
(\lap ( w \conv \Theta^\alpha_i))(s) &= \,W(s) \, (\lap\Theta_i^\alpha)(s)
= \frac{s^{\alpha-1} + \lambda_1^2 s^{-1}}{s^\alpha + \lambda_1^2 - 1} \, 
\frac{1}{s^\alpha + \lambda_i^2}
\\ & =
\frac{1}{\lambda_i^2 - \lambda_1^2+1}\, \left[
\frac{s^{\alpha-1}}{s^\alpha + \lambda_1^2 - 1} 
+ \frac{\lambda_1^2 s^{-1}}{s^\alpha + \lambda_1^2 - 1} 
- \frac{s^{\alpha-1}}{s^\alpha + \lambda_i^2}
- \frac{\lambda_1^2 s^{-1}}{s^\alpha + \lambda_i^2}
\right].
\end{align*}
By taking inverse Laplace transform, we have, for $i\ge 0,$ 
\begin{align*}
& (w \conv \Theta^\alpha_i)(t)  = \frac{1}{\lambda_i^2 - \lambda_1^2+1} \times 
\\  \nonumber & \left[
E_\alpha((1-\lambda_1^2)t^\alpha) 
+ \lambda_1^2 t^\alpha E_{\alpha, \alpha + 1} ((1-\lambda_1^2)t^\alpha)
- E_\alpha(-\lambda_i^2)t^\alpha) 
  - \lambda_1^2 t^\alpha E_{\alpha, \alpha + 1} ((-\lambda_i^2)t^\alpha)
\right].
\end{align*}
Noting that, for $i\ge 1,$  using the shifting identity of the Mittag-Leffler function,
\[(w \conv \Theta^\alpha_i)(t)  = \frac{1}{\lambda_i^2 - \lambda_1^2+1}
\left[
\frac{E_\alpha((1-\lambda_1^2)t^\alpha)}{1-\lambda_1^2} 
+ \frac{\lambda_1^2 - \lambda_i^2}{\lambda_i^2} E_\alpha(-\lambda_i^2)t^\alpha)  
+k_2 - \frac{\lambda_1^2}{\lambda_i^2} \right].
\]

It remains to compute the  convolution terms in $u_{21}$. For convenience,  putting   $k_1=1/(\lambda_1^2-1)$ and $k_2=\lambda_1^2k_1.$ Using  $\lap \{t w(t) \}(s)= -W'(s)$, then    from \eqref{eq:lap-w}, we have
\begin{align*}
\lap \{t w(t) \}(s)
 &=
\frac{k_2}{s^2} + \, \frac{(\alpha - 1 )\, s^{\alpha - 2}}{(s^\alpha 
+ (\lambda_1^2 - 1))^2} - 
\frac{k_1\,s^{2\alpha - 2}}{(s^\alpha + (\lambda_1^2 - 1))^2}.
\end{align*}
Accordingly, 
\begin{align*}
&\lap\{ (t w(t) ) \conv  \Theta^\alpha_1(t) \}(s) 
 = 
\left[ \frac{k_2}{s^2} + \, \frac{(\alpha - 1 ) \, s^{\alpha - 2}}{(s^\alpha + (\lambda_1^2 - 1))^2} - 
\frac{k_1\,s^{2\alpha - 2}}{(s^\alpha + (\lambda_1^2 - 1))^2} \right] 
\frac{1}{s^\alpha + \lambda_1^2}
\\ &= 
\frac{k_2\,s^{-2}}{s^\alpha + \lambda_1^2} + 
[(\alpha-1)s^{\alpha - 2}-k_1s^{2\alpha - 2}] \Big[ \frac{1}{s^\alpha + \lambda_1^2}
- \frac{1}{s^\alpha + (\lambda_1^2 - 1)}
+ \frac{1}{(s^\alpha + (\lambda_1^2 - 1))^2}\Big].
\end{align*}
Thus,
\begin{align*}
(t w(t) ) \conv  \Theta^\alpha_1(t)
& = 
k_2 \,t^{1+\alpha} E_{\alpha, 2+\alpha}(-\lambda_1^2 t^\alpha) 
\\ &
+ (\alpha-1) \left[ 
t\,E_{\alpha,2}(-\lambda_1^2t^\alpha)
- t\,E_{\alpha,2}((1-\lambda_1^2)t^\alpha)
+ t^{\alpha + 1}\,E_{\alpha,2+\alpha}^2((1-\lambda_1^2)t^\alpha) 
\right]
\\ &
- k_1
\left[ t^{1-\alpha}\, E_{\alpha,2 - \alpha}(-\lambda_1^2t^\alpha) 
- t^{1-\alpha}\,E_{\alpha,2 - \alpha}((1-\lambda_1^2)t^\alpha) 
+ t \, E_{\alpha,2}^2((1-\lambda_1^2)t^\alpha) 
\right].
\end{align*}
For the second convolution term in $u_{21}$, since 
\begin{align*}
\lap  \Big\{
w(t) \conv \big(t^{2\alpha - 1} & E_{\alpha,2\alpha}^2(-\lambda_1^2 t^\alpha)\big) 
\Big\}(s) =
k_1 \left[\frac{\lambda_1^2}{s} - \frac{s^{\alpha-1}}{s^\alpha+\lambda_1^2 -1} \right]
\frac{1}{(s^\alpha + \lambda_1^2)^2}
\\ &=
k_1 \left[\frac{\lambda_1^2}{s (s^\alpha + \lambda_1^2)^2} + \frac{ s^{\alpha - 1}}{(s^\alpha + \lambda_1^2)} + \frac{s^{\alpha - 1}}{(s^\alpha + \lambda_1^2)^2} - \frac{s^{\alpha - 1}}{s^{\alpha}+(\lambda_1^2 - 1)}\right],
\end{align*}
\begin{align*}
& w(t) \conv \big(t^{2\alpha - 1} E_{\alpha,2\alpha}^2(-\lambda_1^2 t^\alpha)\big) =
\\  & \quad 
k_1 \left[
\lambda_1^2 \, t^{2\alpha} \, E_{\alpha,2\alpha + 1}^2(-\lambda_1^2t^\alpha) 
+ E_{\alpha}(-\lambda_1^2t^\alpha) 
+ t^{ \alpha}\,E_{\alpha,1+\alpha }^2(-\lambda_1^2 t^\alpha) 
- E_{\alpha}((1-\lambda_1^2)t^\alpha)\right].
\end{align*}

After the above tedious work, we are ready to compare between $w$ and  $u$, 
and their approximations $\tilde w$ and $\tilde u$, respectively. 
To evaluate the errors, we introduce the finer grid
\begin{equation}
\label{eq: fine grid}
\G^N = X\cup
\{\,(t_n+t_{n-1})/2\,:\, 1\leq n \leq N\}\cup \{\,t_n \,:\, 1\leq n \leq N\}.
\end{equation}
($N$ is the number of time mesh subintervals). Thus, for large values of $N$, the error measure 
$ |||v|||_{\G^N}:=\max_{t\in\G^{N}}\|v(t)\|$
approximates the norm $\|v\|_{L_\infty}$ on the time interval $(0,1]$. 

The source term $w$ satisfies the  regularity assumption in \eqref{eq: reg}  for $\sigma=\alpha$. 
Thus,  the theoretical error results   in subsection \ref{sec: error analysis} suggest that $\|w-\tilde w\|_{L^\infty(0,T)}=O((\Delta t)^{\min\{2,\alpha\delta\}})$. 
The numerical numbers in Tables $1$ and $2$ confirm these results for different choices 
of the fractional exponent $\alpha$ and of the graded mesh exponent $\delta$. 
For some graphical illustrations of the errors and the positive influence of the graded mesh, 
see Figures \ref{fig:Sourceterm error delta1} and \ref{fig:Sourceterm error delta3}. 


We focus next on the graphical comparison between the exact solution $u$ and its approximation $\tilde u$.
As mentioned earlier, when the parameter $\gamma<1$, the non-local initial condition
$I^{1-\gamma}u(x,t)|_{t=0}=g(x)=4\cos(4\pi x)$ leads to unbounded solution $u$ as $t \to 0$, 
which is the case here, see \eqref{eq: example exact solution}. 
If we choose $\alpha=\gamma=0.5$, then the two-parameter derivative $\!D^{\alpha,\gamma}$ reduces to 
the Riemann-Liouville derivative $D^\alpha = D^{0.5}$. 
Figure \ref{fig:alf05gam05_largeT}  shows surface plots of both the exact and numerical 
solutions over graded time meshes with mesh exponent $\delta=2$. 
We cut off the initial part of the plots since the solution blows up near $t=0$. 
The surface plots for small $t$ are shown separately in Figure \ref{fig:alf05gam05_smallT}. 
Surface plots for $\alpha=0.5$ and $\gamma=0.7$ are shown in Figures \ref{fig:alf05gam07_largeT} 
and \ref{fig:alf05gam07_smallT}.

When $\gamma=1$, $I^{1-\gamma}u(x,t)|_{t=0}=u(x,t)|_{t=0}=g(x)=4\cos(4\pi x)$. In this case, 
$\!D^{\alpha,\gamma}$ reduces to the Caputo fractional derivative $\cD^\alpha$, 
and the solution is uniformly bounded over the time-space domain. 
Figure \ref{fig:alf0.5gam1} shows the surface plots of $u$ and $\tilde u$ 
over graded time meshes with $\delta=2$.

\section{Concluding Remarks}
A numerical scheme for approximating the time-dependent source term is developed.
The scheme solves the VIE for the source term and approximates the convolutions needed 
to calculate the series coefficients for the solution. 
Due to the singularity of the solution near $t=0$, a graded mesh is used to improve
the convergence rate.
We proved analytically and demonstrated numerically that the rate of convergence is of order 
$\sigma \delta$, where $\sigma$ is the order of regularity of $w$ and $\delta$ is 
the graded mesh exponent.

\section*{Acknowledgment}
The authors would like to acknowledge the support provided by 
King Fahd University of Petroleum \& Minerals.


\newpage

\begin{table}[tbp]
\label{table1}
\centering
	\begin{tabular}{|r|rr|rr|rr|rr|rr|}
			\hline 
$N$ &\multicolumn{2}{c|}{$\delta=1$} & \multicolumn{2}{c|}{$\delta=2$}& 
\multicolumn{2}{c|}{$\delta=3$}& \multicolumn{2}{c|}{$\delta=4$} & \multicolumn{2}{c|}{$\delta=5$} \\
			\hline
			10& 2.37e-02&     & 2.06e-02&     & 1.52e-02&     & 8.72e-03&     & 4.10e-03&      \\
			20& 2.30e-02& 0.04& 1.76e-02& 0.23& 9.30e-03& 0.70& 3.46e-03& 1.34& 1.11e-03& 1.89\\
			40& 2.21e-02& 0.06& 1.38e-02& 0.35& 4.80e-03& 0.96& 1.21e-03& 1.52& 2.81e-04& 1.98\\
			80& 2.10e-02& 0.07& 9.89e-03& 0.48& 2.24e-03& 1.10& 4.04e-04& 1.58& 7.05e-05& 1.99\\
			160&1.97e-02& 0.09& 6.53e-03& 0.60& 1.00e-03& 1.16& 1.34e-04& 1.59& 1.76e-05& 1.99\\
			320&1.81e-02& 0.12& 4.06e-03& 0.68& 4.42e-04& 1.18& 4.43e-05& 1.59& 4.72e-06& 1.94\\
			\hline
	\end{tabular}
\caption {Errors and convergence rates for $\alpha=0.4$ and different values of $\delta$.} 
\end{table}
\begin{table}[tbp]
\label{table2}
\centering
	\begin{tabular}{|r|rr|rr|rr|rr|}
			\hline 
$N$ &\multicolumn{2}{c|}{$\delta=1$} & \multicolumn{2}{c|}{$\delta=2$} 
& \multicolumn{2}{c|}{$\delta=3$} & \multicolumn{2}{c|}{$\delta=4$} \\
			\hline
			10& 2.1240e-02&      & 9.4174e-03&      & 2.0410e-03&      &  6.5174e-04&      \\
			20& 1.8368e-02& 0.210& 3.8912e-03& 1.275& 4.7314e-04& 2.109&  2.3750e-04& 1.456\\
			40& 1.4620e-02& 0.329& 1.4617e-03& 1.413& 1.1443e-04& 2.048&  6.9103e-05& 1.781\\
			80& 1.0625e-02& 0.460& 5.5222e-04& 1.404& 3.0410e-05& 1.912&  1.8453e-05& 1.905\\
			160& 7.1333e-03& 0.575& 2.1308e-04& 1.374& 7.9822e-06& 1.930& 4.7688e-06& 1.952\\
			320& 4.5350e-03& 0.654& 8.3305e-05& 1.355& 2.0444e-06& 1.965& 1.2111e-06& 1.977\\
			\hline
		\end{tabular}
\caption {Errors and convergence rates for $\alpha=0.67$ and different values of $\delta$.}
\end{table}

\begin{figure}[tbp]
\centering
\includegraphics[width=.75\textwidth]{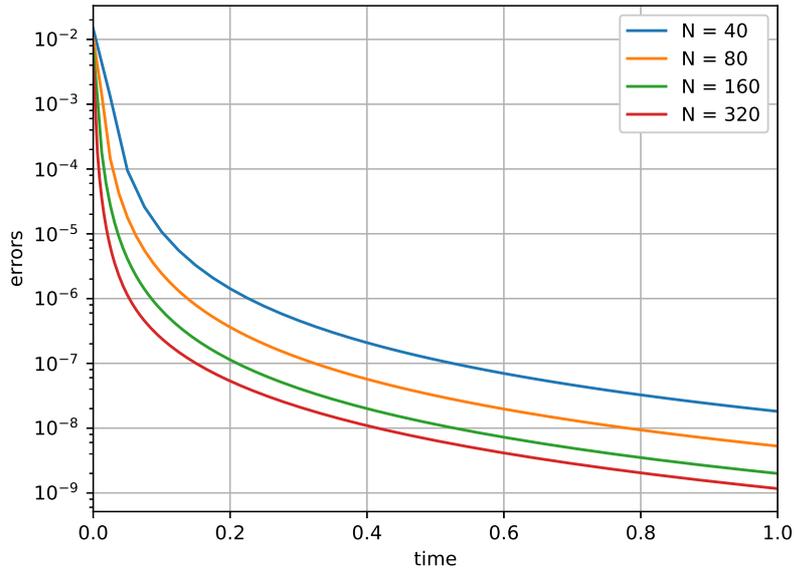}
\caption{Pointwise errors over a uniform mesh for $\alpha = 0.67$.}
\label{fig:Sourceterm error delta1}
\end{figure}
\begin{figure}[tbp]
\centering
\includegraphics[width=.75\textwidth]{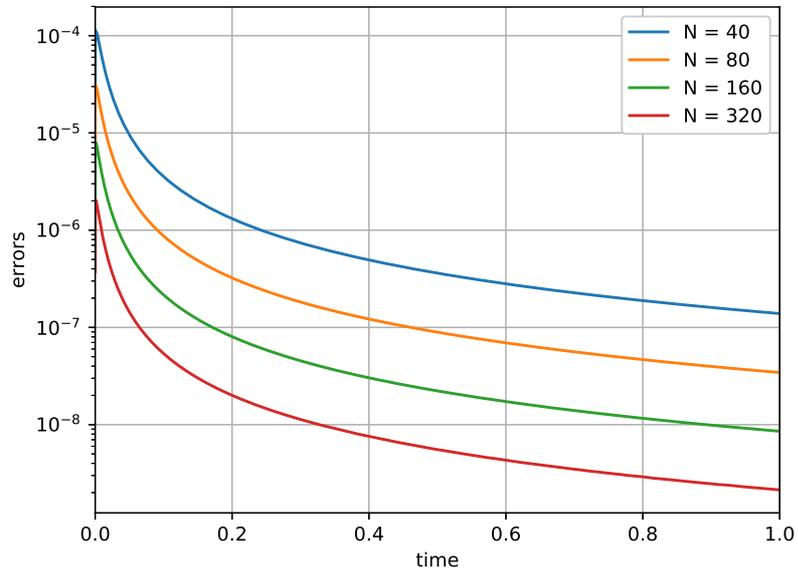}
\caption{Pointwise errors over a nonuniform mesh of $\delta = 3$ for $\alpha = 0.67$.}
\label{fig:Sourceterm error delta3}
\end{figure}

\begin{figure}[tpp]
\centering
\includegraphics[width=.49\textwidth]{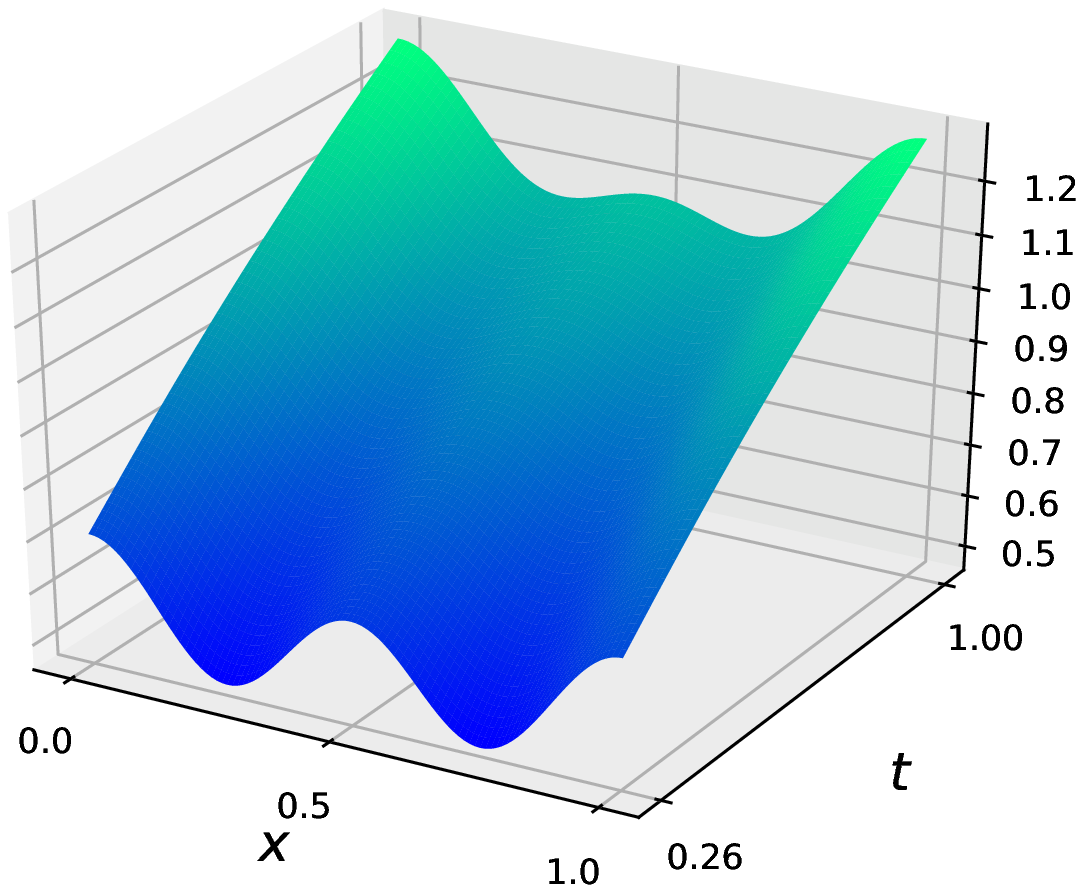} 
\includegraphics[width=.49\textwidth]{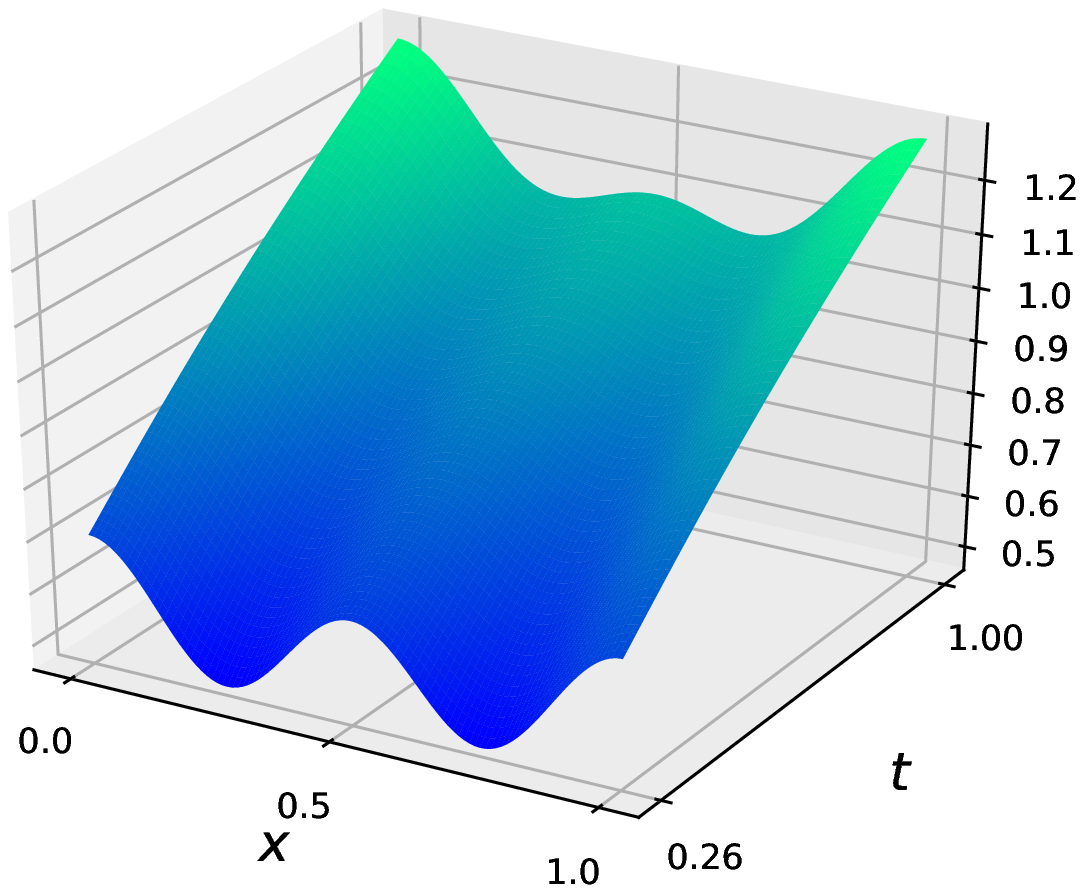}
\caption{Surface plot of the exact solution (left) and the approximate solution (right) 
for $0.26 < t < 1.0$, $\alpha=\gamma=0.5$ and $\delta=2$.}
\label{fig:alf05gam05_largeT}
\end{figure}

\begin{figure}[tbp]
\centering
\includegraphics[width=.49\textwidth]{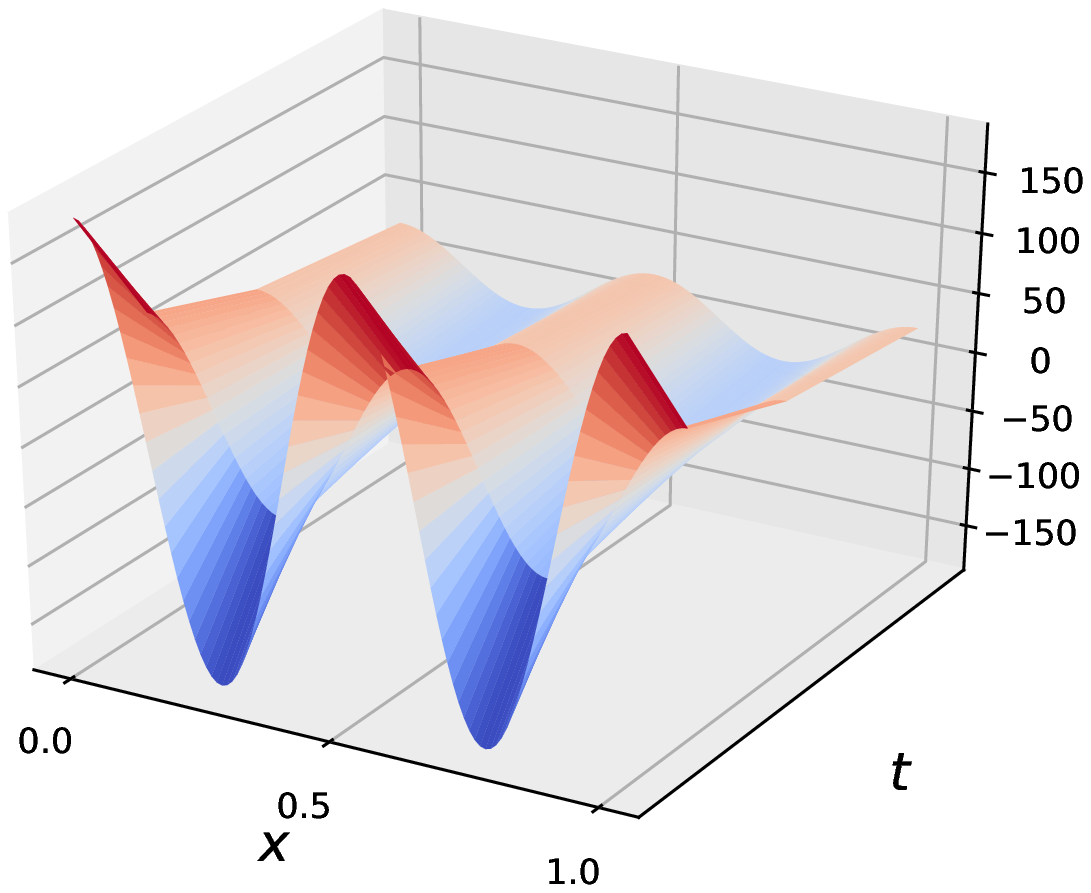} 
\includegraphics[width=.49\textwidth]{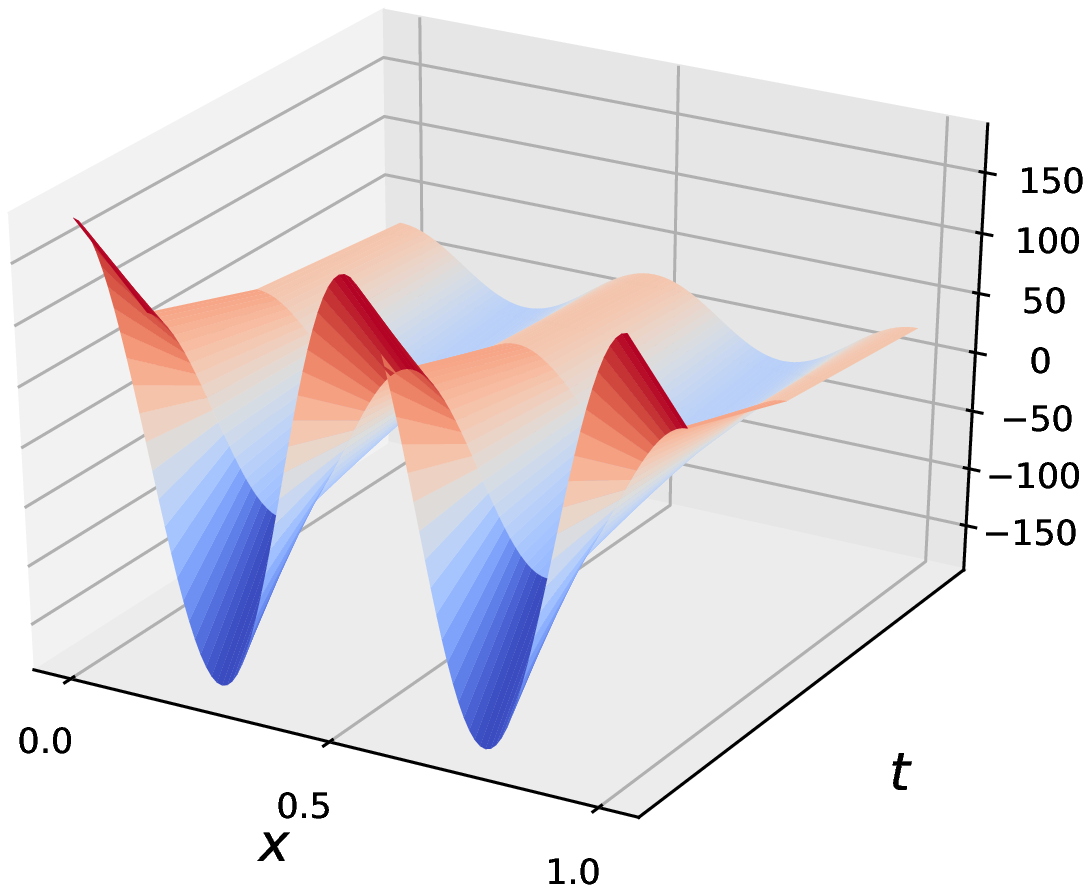}
\caption{Surface plot of the exact solution (left) and the approximate solution (right) for 
$0.1 \times 10^{-3} < t < 1.6\times 10^{-3}$, $\alpha=\gamma=0.5$ and $\delta=2$.}
\label{fig:alf05gam05_smallT}
\end{figure}

\begin{figure}[tbp]
\centering
\includegraphics[width=.49\textwidth]{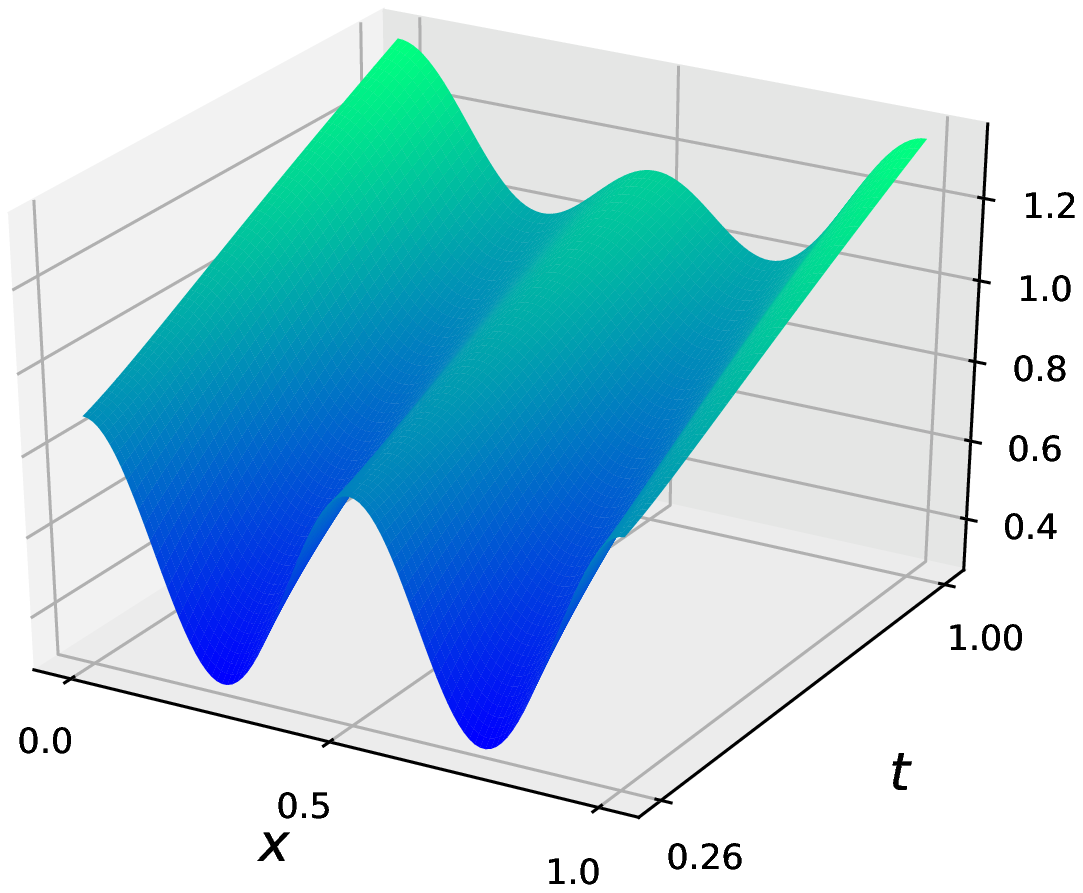} 
\includegraphics[width=.49\textwidth]{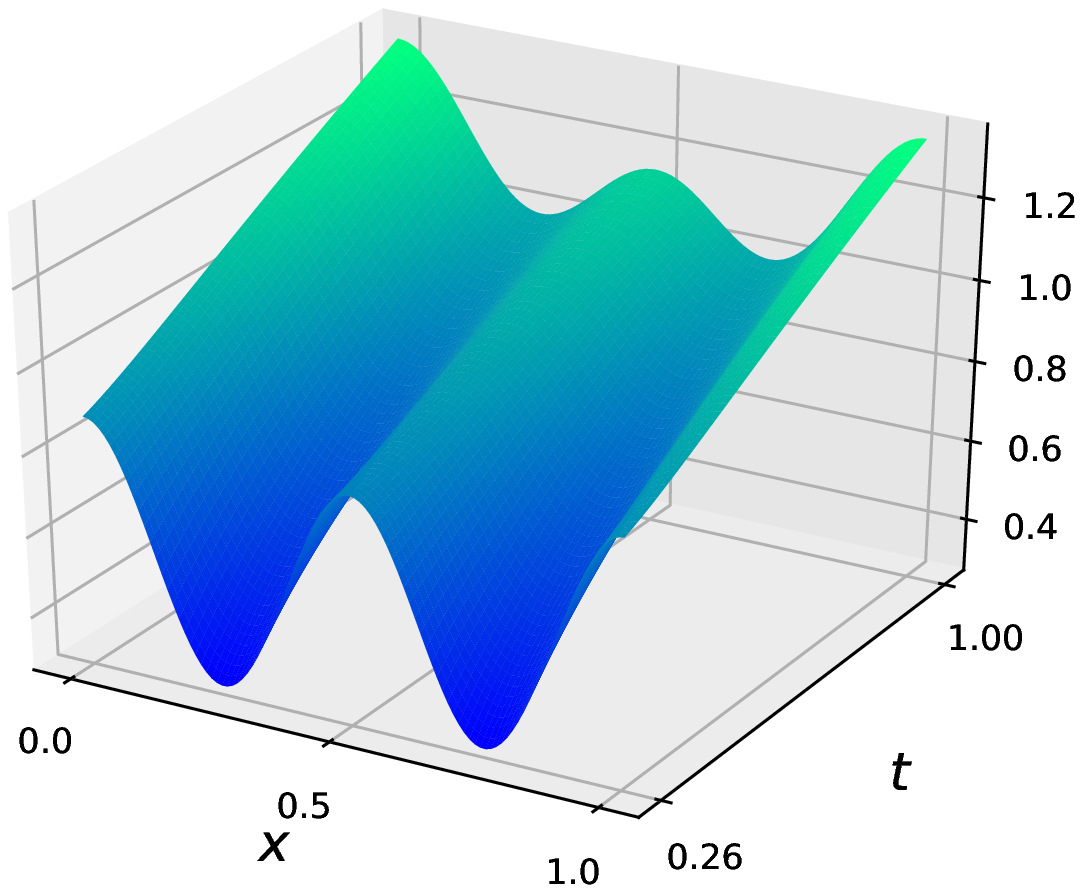}
\caption{Surface plot of the exact solution (left) and the approximate solution (right) 
for $0.26 < t < 1$, $\alpha=0.5$ and $\gamma=0.7$.}
\label{fig:alf05gam07_largeT}
\end{figure}

\begin{figure}[tbp]
\includegraphics[width=.49\textwidth]{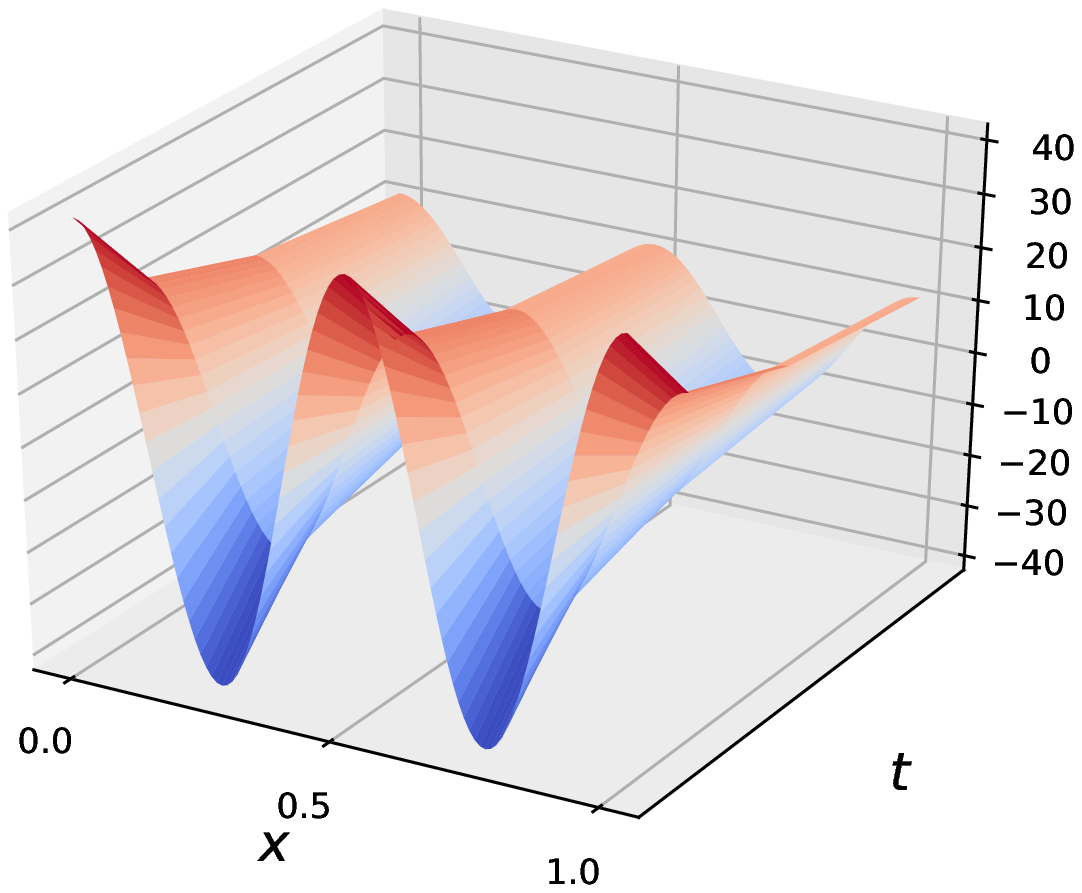} 
\includegraphics[width=.49\textwidth]{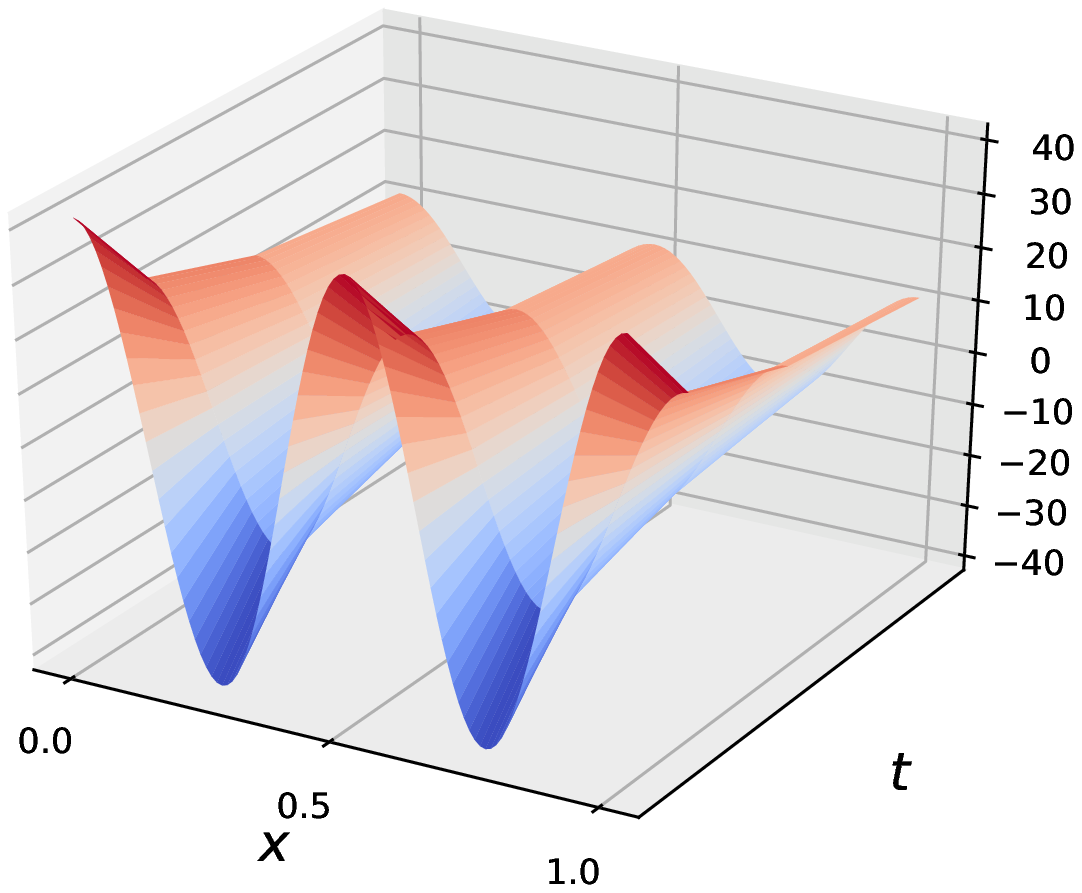}
\caption{Surface plot of the exact solution (left) and the approximate solution (right) 
for $0.1 \times 10^{-3} < t < 1.6\times 10^{-3}$,
$\alpha=0.5$ and $\gamma=0.7$.}
\label{fig:alf05gam07_smallT}
\end{figure}

\begin{figure}[tbp]
\centering
\includegraphics[width=.49\textwidth]{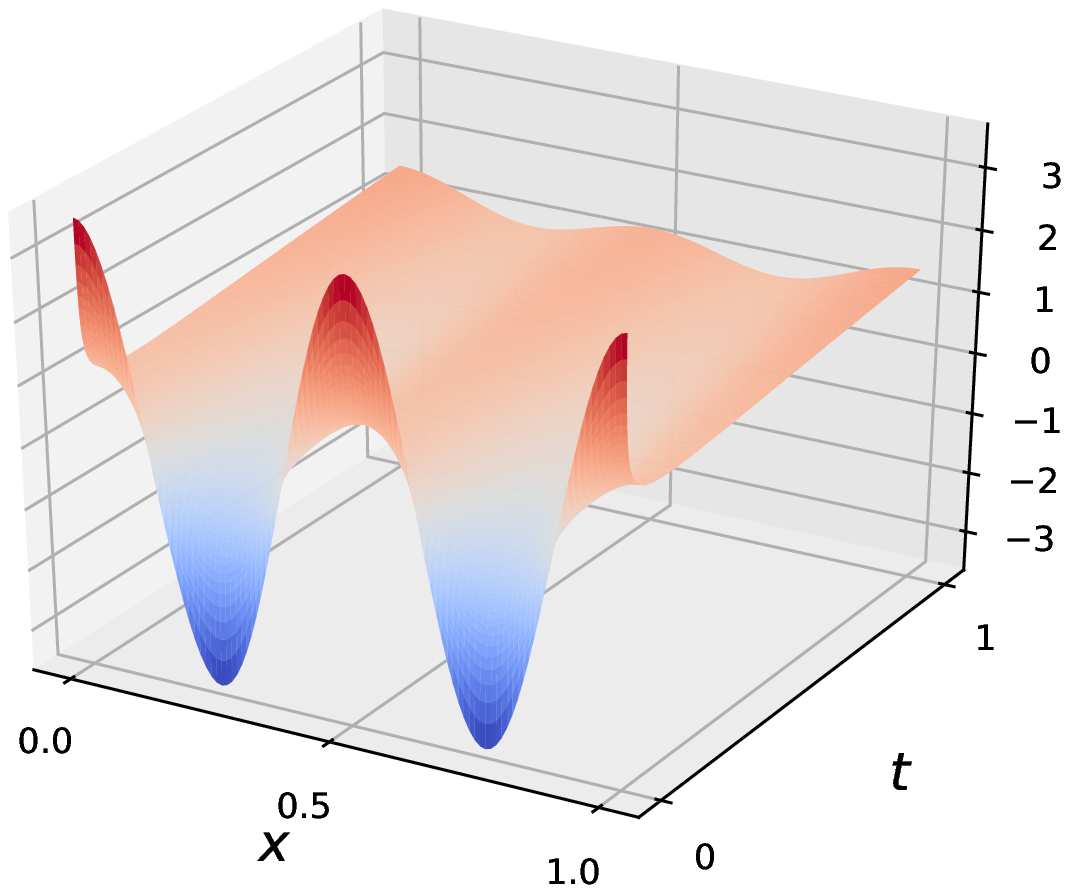} 
\includegraphics[width=.49\textwidth]{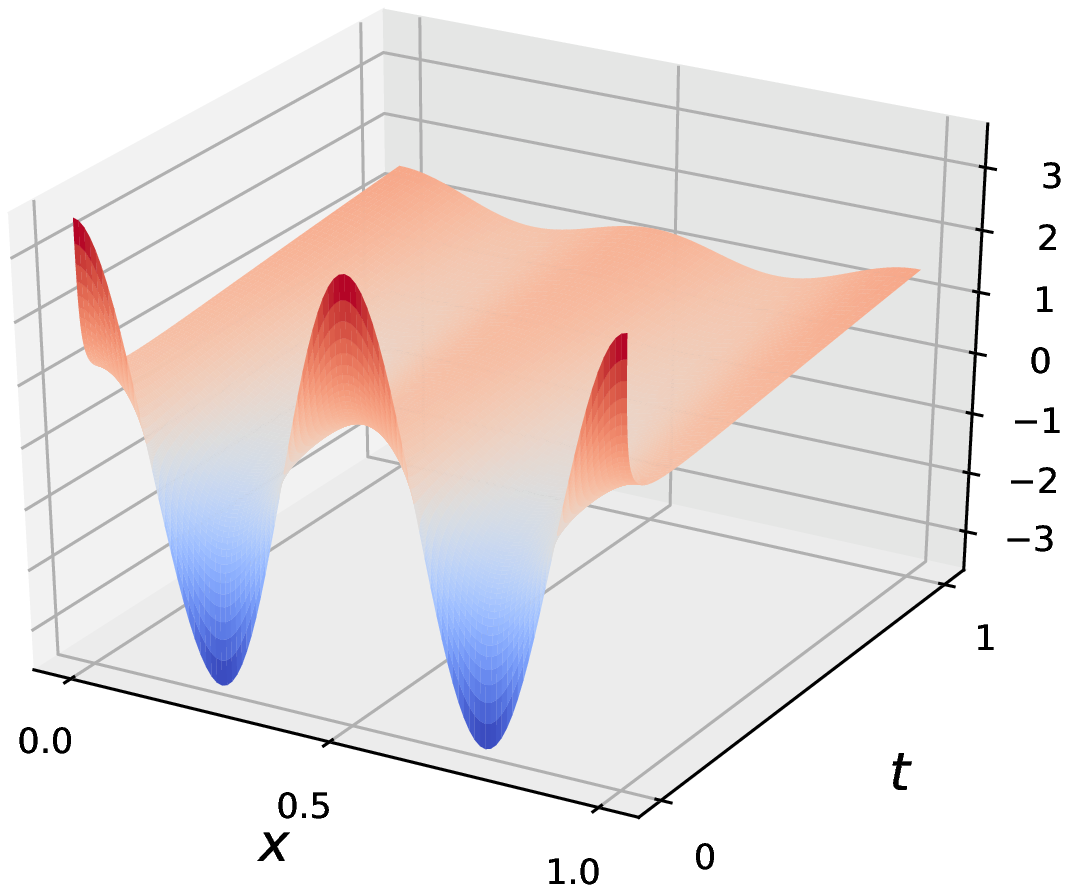}
\caption{Surface plot of the exact solution (left) and the approximate solution (right) for 
$\alpha=0.5$, $\gamma=1$ and $\delta =2$.}
\label{fig:alf0.5gam1}
\end{figure}

\end{document}